%% file: GLN.tex
\documentclass[draft]{article}

\input{xypic}
\xyoption{all}

\usepackage{mathpazo}

\newcommand{\Pc}{F}

\input{MC}

\begin{document}

\title{The Motivic Cohomology of Stiefel Varieties}
\author{Ben Williams}

\maketitle

\begin{abstract} The main result of this paper is a computation of the motivic cohomology of varieties of
  $n \times m$-matrices of of rank $m$, including both the ring structure and the action of the
  reduced power operations. The argument proceeds by a comparison of the general linear group-scheme
  with a Tate suspension of a space which is $\A^1$-equivalent to projective $n-1$-space
  with a disjoint basepoint.

\noindent

Secondary: 20G20, 57T10.
\end{abstract}

\section{Introduction}
The results of this paper are theorem \ref{th:mainch1}, which computes the motivic cohomology (or
higher Chow groups) of varieties of $n \times m$-matrices of of rank $m$, the Stiefel varieties of
the title, including the ring structure, and theorems \ref{c:SqOnW}, \ref{c:POnW} which compute the
action of the motivic reduced power operations, the first for the Steenrod squares, the second for
odd primes. The ring structure on $H^{*,*}(\Gl(n))$ has already been computed, in \cite{PUSHIN},
where the argument is via comparison with higher $K$-theory. We are able to offer a different
computation of the same, which is more geometric in character, relying on an analysis of comparison
maps of varieties rather than of cohomology theories. 

The equivalent computation for \'etale cohomology appears in \cite{MRaynaud}; the argument there is
by comparison with singular cohomology, and that suffices to determine even the action of the
reduced power operations on the \'etale cohomology of $\GL(n, \Z)$, and from there on the \'etale
cohomology of $\GL(n, k)$ where $k$ is an arbitrary field. The universal rings of loc.~cit.~are not
quite Stiefel varieties, but they are affine torsors over them. In principle, by slavish imitation
of loc.~cit.~the methods of the present paper allow one to prove that with sporadic exceptions the
universal stably-free module of rank at least $2$ over a field is not free, and to do so without recourse to a
non-algebraic category.

In this paper we deduce, by elementary means, the additive structure of the cohomology of Stiefel
varieties, proposition \ref{p:CohW1}, then we deduce the effect on cohomology of two comparison maps
between the different Stiefel manifolds, these are propositions \ref{p:Wincl} and
\ref{p:GLNsurj}. By use of these, the ring structure and the reduced power operations in all cases
may be deduced from the case of $\Gl(n)$. We present a map in homotopy $\G_m \wedge \P^{n-1}_+ \to \Gl(n)$ that
is well-known in the classical cases of $\R^* \wedge \R P^{n-1}_+ \to O(n)$ and $\C^* \wedge
\C P^{n-1}_+ \to U(n)$, \cite[Chapter 3]{JAMES}, where it fits into a larger pattern of maps from
suspensions of so-called ``stunted quasiprojective spaces'' into Stiefel manifolds. In the classical
case, with some care, one may show that the cohomology of a Stiefel manifold is generated as a ring
by classes detected by maps to stunted projective spaces. The comparison of Stiefel manifolds with
the appropriate stunted projective-spaces underlies much of the classical homotopy-theory of Stiefel
manifolds, the theory is presented thoroughly in \cite{JAMES}, we mention in addition only the
spectacular resolution of the problem of vector fields on spheres in \cite{ADAMSVBL}. Care is required even in
cohomology calculations because the products of classes in the cohomology of Stiefel manifolds muddy
the water. In $\A^1$-homotopy, there are also analogues of stunted projective spaces, but the
bigrading on motivic cohomology means that the products of generating classes may be disregarded in
articulating the range in which the cohomology $H^{*,*}(\G_m \wedge \P^{n-1}_+)$ and
$H^{*,*}(\Gl(n))$ coincide, and we arrive at a much simpler statement, theorem \ref{t:MainComp}
without having to mention the stunted spaces at all.

In proving theorem \ref{t:MainComp}, we employ a calculation in higher Chow groups. This
calculation is on the one hand attractive in its geometric and explicit character but on the other
hand it is the chief obstacle to extending the scope of the arguments presented here to other
theories than motivic cohomology.

\section{Preliminaries}

We compute motivic cohomology as a represented cohomology theory in the motivic- or $\A^1$-homotopy category of
Morel \& Voevodsky, see \cite{MV} for the construction of this category. The best reference for the
theory of motivic cohomology is \cite{MVW}, and the proof that the theory presented there is
representable in the category we claim can be found in \cite{DELIGNE}, subject to the restriction
that the field $k$ is perfect. Motivic cohomology, being a cohomology theory (again at least when
$k$ is perfect) equipped with suspension isomorphisms for both suspensions, $\Sigma_s$, $\Sigma_t$,
is represented by a motivic spectrum, of course, but we never deal explicitly with such objects.

We therefore fix a perfect field $k$. We shall let $R$ denote a fixed commutative ring of
coefficients.
 We denote the terminal object, $\spec k$, by $\pt$. If $X$ is a finite
type smooth $k$-scheme or more generally an element in the category $s\cat{Shv}_\Nis(\Sm/k)$), we write
$H^{*,*}(X;R)$ for the bigraded motivic cohomology ring of $X$. This is graded-commutative in
the first grading, and commutative in the second.  If $R \to R'$ is a ring map, then there is a map
of algebras $H^{*,*}(X;R) \to H^{*,*}(X;R')$. It will be of some importance to us that most of our
constructions are functorial in $R$, when the coefficient ring is not specified, therefore, it is to
be understood that arbitrary coefficients $R$ are meant and that the result is functorial in
$R$.

 We write $\M_R$ for the ring $H^{*,*}(\pt;R)$. Since
$\pt$ is a terminal object, the ring $H^{*,*}(X;R)$ is in fact an
$\M_R$ algebra. We assume the following vanishing results for a $d$-dimensional smooth scheme $X$:
$H^{p,q}(X;R) = 0$ when $p>2q$, $p > q+ d$ or $q< 0$. 

 At one point we employ the comparison theorem relating motivic cohomology and the higher Chow
groups. For all these, see \cite{MVW}.

Let $\Sm/k$ denote the category of smooth $k$-schemes, and let $\cat{Aff}/k$ denote the category of
affine regular finite-type $k$-schemes. We shall frequently make use of the following version of the Yoneda lemma
\begin{lemma} \alabel{l:Yoneda}
  The functor $\Sm/k \to \Pre (\cat{Aff}/k)$ given by $X \mapsto h_X$, where $h_X(\Spec R)$ denotes
  the set of maps $\Spec R \to X$ over $\Spec k$ is
  a full, faithful embedding.
\end{lemma}
\begin{proof}
  The standard version of Yoneda's lemma is that there is a full, faithful embedding $\Sm/k \to
  \Pre(\Sm/k)$. The functor we are considering is the composition $\Sm/k \to \Pre(\Sm/k) \to
  \Pre(\cat{Aff}/k)$ obtained by restricting the domain of the functors in $\Pre (\Sm/k)$. One can
  write any $Y$ in $\Sm/k$ as a colimit of spectra of finite-type
  $k$-algebras. We have
  \begin{equation*}
    \Sm/k(Y ,X) = \Sm/ k (\colim \Spec A_i, X) = \lim \Sm /k (\Spec A_i ,X) = \lim X(A_i)
  \end{equation*}
  from which it follows that the functor $\Sm / k \to \cat{Aff}/k$ inherits fullness and fidelity from
  the Yoneda embedding by abstract-nonsense arguments.
\end{proof}
We shall generally write $X(R)$ for  $h_X(\Spec R)$.

In practice this result means that rather than specifying a map of schemes $X \to Y$ explicitly, we shall happily
exhibit a set-map $X(R) \to Y(R)$, where $R$ is an arbitrary finite-type $k$-algebra, and then
observe that this set map is natural in $R$. The result is a map in $\Pre(\cat{Aff}/k)$, which is
therefore (by the fullness and fidelity of Yoneda) also understood as a map of schemes $X \to Y$.

\section{The Additive Structure} 

\begin{prop} \alabel{p:BDIFF} Let $X$ be a smooth scheme and suppose $E$ is an $\A^n$-bundle over
  $X$, and $F$ is a sub-bundle with fiber $\A^\ell$, then there is an exact triangle of
   graded $H^{*,*}(X)$-modules
  \begin{equation*}
    \xymatrix{ H^{*,*}(X)\tau\ar^j[rr] & & H^{*,*}(X) \ar[dl] \\ & H^{*,*}(E\sm F) \ar^{\bd}[ul] }
  \end{equation*}
  where $|\tau|= (2n-2\ell, n -\ell)$
\end{prop}
\begin{proof}
  This is the localization exact triangle for the closed sub-bundle $F \subset E$, where 
  \[ H^{*,*}(E) = H^{*,*}(F)  = H^{*,*}(X) \]
 arising from the cofiber sequence (see~\cite{MV} for this and
  all other unreferenced assertions concerning $\A^1$-homotopy)
  \begin{equation*}
    \xymatrix{    E \sm F \ar[r] & E \ar[r] & \operatorname{Th} (N)}
  \end{equation*}
  where $N$ is the normal-bundle of $F$ in $E$. There are in two possible choices for $\tau$, since
  $\tau$ and $-\tau$ serve equally well. We make the convention that in any localization exact
  sequence associated with a closed immersion of smooth schemes $Z \to X$, viz.
  \[ \xymatrix{ \ar[r] & H^{*,*}(Z)\tau \ar[r] & H^{*,*}(X) \ar[r] & H^{*,*}(X \sm Z) \ar[r] & } \]
  the class $\tau$ should be the class which, under the natural isomorphism of the above with the
  localization sequence in higher Chow groups, corresponds to  the class represented by $Z$ in $CH^*(Z,*)$.

  In the case where $F=X$ is the zero-bundle, then $j$ takes $\tau$ to $e(E)$, the Euler
  class, as proved in \cite{VREDPOWER}. In general, by identifying $H^{*,*}(X)\tau$ with
  $CH^{*-n+\ell}(F,*)$, the higher Chow groups of $F$ as a closed subscheme of $E$, and employing
  covariant functoriality of higher Chow groups for the closed immersions $X \clsub F \clsub E$, we
  see that $j(\tau)e(F) = e(E)$.
\end{proof}

As shall be the case throughout, $H^{*,*}(X)$ denotes cohomology with unspecified coefficients, $R$, and
the result is understood to be natural in $R$. For the naturality of the localization sequence in
$R$, one simply follows through the argument in \cite{MVW}, which reduces it to the computation of
$H^{*,*}(\P^d) = \M_R[\theta]/(\theta^{d+1})$ which is natural in $R$ by elementary means, c.f.~\cite{FULTON}.

We note that $|j(\tau)| = (2c,c)$, so if, as often happens,
$H^{2c, c}(X) = 0$, this triangle is a short exact sequence of
$H^{*,*}(X)$-modules:
\begin{equation*}
  \xymatrix{ 0 \ar[r] & H^{*,*}(X) \ar[r] & H^{*,*}(E\sm F) \ar[r] & H^{*,*}(X) \rho \ar[r] & 0 }
\end{equation*}
Here $|\rho| =( 2n -2\ell - 1, n- \ell)$

Since $H^{*,*}(X)\rho$ is a free graded $H^{*,*}(X)$-module, this short exact sequence of graded modules splits,
 there is an isomorphism of $H^{*,*}(X)$-modules \[H^{*,*}(E\sm F) \isom H^{*,*}(X) \oplus H^{*,*}(X)\rho\]

We remark that $|\rho| = (2c-1,c)$, so that $2\rho^2 = 0$ by anti-commutativity, we now see
that $(a+b\rho)(c+d\rho) = ac + (ad + (-1)^{\deg{c}}bc)\rho + (-1)^{\deg{d}}bd \rho^2$, so in many
cases (e.g. when $1/2 \in R$) the multiplicative structure is fully determined, and $H^{*,*}(E \sm F) =
{H^{*,*}(X)}[\rho]/(\rho^2)$

Observe that if $H^{2n,n}(X) = 0$ for $n>0$, as often happens, then the same applies to
$H^{*,*}(E\sm F)$.

We will have occasion later to refer to the following two results, which appear here for want of
anywhere better to state them

\begin{prop}
  Let $X$ be a scheme and suppose $E$ is a Zariski-trivializeable fiber bundle with fiber $F \weq
  \pt$. Then $E \weq X$.
\end{prop}
\begin{proof}
  This is standard, see \cite{DHI}.
\end{proof}

When we use the term `bundle', we shall mean a Zariski-trivializable bundle over a scheme.
The following two propositions allow us to identify affine bundles which are not necessarily vector-bundles.

\begin{prop} \alabel{p:projBundleTheorem}
  Suppose $X$ is a scheme, $P$ is a projective bundle of rank $n$ over $X$ and $Q$ is a projective subbundle of
  rank $n-1$, Then $P \sm Q$ is a fiber bundle with fiber $\A^{n-1}$.
\end{prop}
\begin{proof}
  This follows immediately by considering a Zariski open cover trivializing both bundles.
\end{proof}

Let $k$ be a field. Let $W(n,m)$ denote the variety of full-rank $n\times m$ matrices over
$k$, that is to say it is the open subscheme of $\A^{nm}$ determined by the nonvanishing of at least
one $m \times m$-minor. Without loss of generality, $m\le n$. By a \textit{Stiefel Variety\/} we mean such a variety
$W(n,m)$.

\begin{prop} \alabel{p:CohW1}
  The cohomology of $W(n,m)$ has the following presentation as an $\M_R$-algebra:
  \begin{equation*}
    H^{*,*}(W(n,m);R) = \frac{\M_R[\rho_n,\dots, \rho_{n-m+1}]}{I} \qquad |\rho_i| = (2i-1,i)
  \end{equation*}
  The ideal $I$ is generated by relations $\rho_i^2-a_{n,i}\rho_{2i-1}$, where the elements
  $a_{n,i}$ lie in $\M_R^{1,1}$ and satisfy $2a_{n,i}=0$.
\end{prop}

We shall later identify $a_{n,i}$ as $\{-1\}$, the image of the class of $-1$ in $k^* = H^{1,1}(\Spec
k;\Z)$ under the map $H^{*,*}(\Spec k; \Z) \to H^{*,*}(\Spec k; R)$.

\begin{proof}
  If $n=1$, there is only one possibility to consider, that of $W(1,1) = \A^1 \sm \{0\}$, the
  cohomology of which is already known from \cite{VREDPOWER}, and is as asserted in the proposition. We therefore assume $n \ge 2$.

  The proof proceeds by induction on $m$, starting with $m=1$ (we could start with $W(n,0)
  = \pt$). In this case $W(n,1) = \A^n  \sm \{0\}$, and $H^{*,*}(W(n,1)) = \M[\rho_n]/(\rho_n^2)$.

  $W(n,m-1)$ is a dense open set of $\A^{nm}$, and as such is a smooth scheme. If $m < n$,
  there is a trivial $\A^n$-bundle over $W(n,m-1)$, the fiber over a matrix $A$ is the set
  of all $n \times m$-matrices whose first $m-1$ columns are the matrix $A$ 
  \begin{equation*}
    \begin{pmatrix} &  & & v_1 \\ & A & & \vdots \\ & & & v_n \end{pmatrix}
  \end{equation*}
  As a sub-bundle of this bundle, we find a trivial $\A^{m-1}$-bundle; the fiber of which over a
  $\overline{k}$-point (i.e.~a matrix) $A$ is the set of matrices where $(v_1, \dots, v_n)$ is in the
  row-space of $A$. Proposition \ref{p:BDIFF} applies in this setting, and we conclude that there
  exists an exact triangle
  \begin{equation*}
    \xymatrix{ H^{*,*}(W(n,m-1))\tau \ar[rr] && H^{*,*}(W(n,m-1)) \ar[dl] \\ & H^{*,*}(W(n,m))
      \ar^{\bd}[ul] &}
  \end{equation*}
  
  Since $H^{2i,i}(W(n,m-1))=0$ by induction, so this triangle splits to give 
  \[ H^{*,*}(W(m,n)) \isom  H^{*,*}(W(m-1,n)) \oplus H^{*,*}(W(m-1,n))\rho_{n-m+1} \]
 where $|\rho_{n-m+1}| = (2(n-m+1)-1,n-m+1)$. By graded-commutativity we have, $2\rho_{n-m+1}^2 = 0$.
 By considering the bigrading on the motivic cohomology, and the vanishing results, we know that
  \begin{equation*}
     \rho_{r}^2 \in H^{4r-2,2r}(W(n,m);R) = H^{4r-2, 2r}(W(n,m-1);R) \text{ where $r=n-m+1$}
  \end{equation*}
We can describe
  $H^{*,*}(W(n,m-1);R)$ as an $\M$-module as follows
  \begin{equation*}
    H^{*,*}(W(n,m-1);R) \isom \M \oplus \bigoplus_i \M\rho_i \oplus \bigoplus_{i<j} \M\rho_i \rho_j
    \oplus J
  \end{equation*}
  where $J$ is the submodule generated by multiples of at least three distinct classes of the form $\rho_i$.
  Since $\weight(\rho_i\rho_j) = i+j > 2(n-m+1)$ for all $i,j \geq n-m+2$, it follows the higher product
  terms are irrelevant to the determination of the cohomology group $H^{4n-4m+2,2n-2m+2}(W(n,m-1;R))$. 

 There are two possibilities to consider. First, that $n>2m-1$, which by consideration of the
 grading forces $H^{4n-4m+2, 2n-2m+2}(W(n,m-1);R) = 0$, and so $\rho_{n-m+1}^2 =0$. The other is $n
 \le 2m-1$, in which case
\[H^{4n-4m+2, 2n-2m+2}(W(n,m-1);R) = \M^{1,1}_R\rho_{2n-2m+1}\]
so that
 $\rho_{n-m+1}^2 = a_{n,m}\rho_{2n-2m+1}$ as required. The bigrading alluded to above forces
 $2a_{n,m} = 0$.
\end{proof}

We denote the cohomology ring 
\begin{equation*} H^{*,*}(W(n,m);R) = \M_R[\rho_n, \dots, \rho_{n-m+1}]/I
\end{equation*}
where the ideal $I$ is understood to depend on $n$, $m$. We shall need the following technical lemma

\begin{lemma} \alabel{l:techbff} Let $Z \to X$ be a closed immersion of irreducible smooth schemes,
  and let $f: X' \to X$ be a map of smooth schemes such that $f^{-1}(Z)$ is again smooth and
  irreducible and so that is either $f$ is flat or split by a flat map, in the sense that there
  exists a flat map $s: X \to X'$ such that $s \circ f = \id_{X'}$. Then there is a map of
  localization sequences in motivic cohomology
  \[ 
  \xymatrix{  \ar[r] &   H^{*,*}(Z)\tau\ar[d] \ar[r] & H^{*,*}(X) \ar[r] \ar[d] & \ar[d]
    \ar[r]  H^{*,*}(X \sm Z) & \\
    \ar[r] &  H^{*,*}(f^{-1}(Z))\tau' \ar[r] & H^{*,*}(X') \ar[r] & 
    \ar[r] H^{*,*}(X' \sm f^{-1}(Z)) &  }
    \]
such that the last two vertical arrows are the functorial maps on cohomology and such that $\tau
\mapsto \tau'$.
\end{lemma}
The giving of references for results concerning higher Chow groups is deferred to the beginning of section \ref{s:HIT}.

\begin{proof}
  One begins by observing the existence in general of the following diagram
  \begin{equation*}
    \xymatrix{ X' \sm f^{-1}(Z) \ar[r] \ar[d] & X' \ar[r] \ar[d]  & \operatorname{Th} N_{f^{-1}(Z)
        \to X'} \ar@{-->}[d] \\ 
      X \sm Z\ar[r] & X \ar[r]  & \operatorname{Th} N_{Z
        \to X} }
  \end{equation*}
  where the dotted arrow exists for reasons of general nonsense. 

  There is in general a map on cohomology arising from the given diagram of cofiber sequences, but
  we cannot at this stage predict the behavior of the map induced by the dotted arrow.
   When the map $X' \to X$ is flat, the pullback $f^{-1}(Z) \to Z$ is too. We identify the motivic
    cohomology groups with the higher Chow groups, giving the localization sequence
  \[ 
  \xymatrix{  \ar[r] &   CH^{*}(Z, *)\ar@{-->}[d] \ar[r] & CH^{*}(X, *) \ar[r] \ar[d] & \ar[d]
    \ar[r]  CH^{*}(X \sm Z, *) & \\
    \ar[r] &  CH^{*}(f^{-1}(Z), *) \ar[r] & CH^{*}(X', *) \ar[r] & 
    \ar[r] CH^{*}(X' \sm f^{-1}(Z), *) & }
    \]
    and in this case $\tau$, $\tau'$ become the classes of the cycles $[Z]$, $[f^{-1}Z]$. Since the map
    \[CH^*(Z,*) \to CH^*(f^{-1}(Z),*)\] 
    is the contravariant map associated with pull-back along a
    flat morphism, it follows immediately that $\tau \mapsto \tau'$.
\end{proof}

The following results are analogues of classically known facts.

\begin{prop}\alabel{p:Wincl}
  For $m'\le m$, there is a projection $W(n,m) \to W(n,m')$ given by omission of the last
  $m-m'$-vectors. On cohomology, this yields an inclusion 
  \begin{equation*}
    \M(\rho_n, \dots,
  \rho_{n-m'+1})/I \to \M(\rho_n, \dots,\rho_{n-m'+1}, \dots, \rho_{n-m+1})/I
  \end{equation*}
\end{prop}
\begin{proof}
  It suffices to prove the case $m'=m-1$. In this case, the map $W(n,m) \to W(n,m-1)$ is
  the fiber bundle from which we computed the cohomology of $W(n,m)$, and the result on
  cohomology holds by inspection of the proof.
\end{proof}

\begin{prop} \alabel{p:GLNsurj}
  Given a nonzero rational point, $v \in (\A^n \sm 0)(k)$, and a complementary $n-1$-dimensional subspace $U$ such
  that $\lspan{v} \oplus U = \A^n (k)$, there is a closed immersion $\phi_{v,U}: W(n-1,m-1) \to W(n,m)$
  given by identifying $W(n-1,m-1)$ with the space of independent $m-1$-frames in $U$, and then
  prepending $v$. On cohomology, this yields the surjection $\M[\rho_n, \dots, \rho_{n-m+1}]/I \to
  \M[\rho_{n-1}, \dots, \rho_{n-m+1}]/I$ with kernel $(\rho_n)$.
\end{prop}
\begin{proof}
  We prove this by induction on the $m$, which is to say we deduce the case $(n,m)$ from the case
  $(n, m-1)$. The base case of $m=1$ is straightforward.


  Recall that we compute the cohomology of $W(n,m)$ by forming a trivial bundle $E_{n,m} \weq
  W(n,m-1)$ over $W(n,m-1)$, which
  on the level of $R$-points  consists of
  matrices
  \begin{equation*}
    \begin{pmatrix} &  & & v_1 \\ & A & & \vdots \\ & & & v_n \end{pmatrix}
  \end{equation*}
  and removing the trivial closed sub-bundle $Z_{n,m}$ where the vector $(v_1,\dots,v_n)$ is in the span of the
  columns of $A$. There is then an open
  inclusion 
  \[ E_{n,m} \sm Z_{n,m} \isom  W(n,m) \to E_{n,m} \weq W(n,m-1)\]

  The inclusion $\phi_{v,U}: W(n-1,m-1) \to W(n,m)$, without loss of generality can be assumed to act on field-valued points as as
  \begin{equation*}
   \xymatrix{ B \ar@{|->}[rr]^{\phi_{v,U}} & & {\displaystyle \begin{pmatrix} 1 & 0 \\ 0 & B \end{pmatrix}}}
  \end{equation*}
  We abbreviate this map of schemes to $\phi$, and denote the analogous maps $W(n-1,m-2) \to W(n,
  m-1)$, $Z_{n-1,m-1} \to Z_{n,m}$, $E_{n-1,m-1} \to E_{n,m}$ etc.~also by $\phi$ by abuse of
  notation. The following are pull-back diagrams
  \begin{equation*}
    \xymatrix{E_{n-1,m-1} \ar[r] & E_{n,m} \\ Z_{n-1,m-1} \ar[u] \ar[r] & Z_{n,m} \ar[u] } \quad \xymatrix{ E_{n-1,m-1} \ar[r] & E_{n,m}
      \\ W(n-1,m-1) \ar[u] \ar^{\phi}[r] & W(n,m) \ar[u] }
  \end{equation*}
  The second square above is homotopy equivalent to 
  \begin{equation*}
    \xymatrix{ W(n-1,m-2) \ar^{\phi}[r] \ar[d] & W(n,m-1) \ar[d] \\ W(n-1,m-1) \ar^{\phi}[r] & W(n,m) }
  \end{equation*}
  from which we deduce that the map
  \[\phi^*: H^{*,*}(W(n,m)) \to  H^{*,*}(W(n-1,m-1))\] satisfies $\phi^*( \rho_j) = \rho_j$ for $n-m + 2 \le j \le n-1$ and
  $\phi^*(\rho_n) = 0$, since this holds for $W(n-1,m-2) \to W(n,m-1)$ by induction. 

  The hard part is the behavior of the element $\rho_{n-m+1}$, which is in the kernel of \[H^{*,*}(W(n,m)) \to
  H^{*,*}(W(n,m-1))\]
  Recall that $\rho_{n-m+1} \in H^{*,*}(W(n,m))$ is the preimage of the Thom
  class $\tau$ under the map 
  \[\partial: H^{*,*}(W(n,m)) \to H^{*,*}(W(n,m-1))\tau = H^{*,*}(Z_{n,m})\tau\]
 
  We should like to assert that the map of localization sequences
 \begin{equation} \alabel{e:mcchris} 
    \xymatrix@C=18pt{ \ar[r] & H^{*,*}(W(n,m)) \ar^{\phi^*}[d] \ar^{\partial}[r] & H^{*,*}(Z_{n,m}) \tau
      \ar[d] \ar[r] & H^{*,*}(E_{n,m})\ar[d] \ar[r] & \\ 
    \ar[r] & H^{*,*}(W(n-1,m-1)) \ar^{\partial}[r] & H^{*,*}(Z_{n-1,m-1}) \tau'
       \ar[r] & H^{*,*}(E_{n-1,m-1})  \ar[r] & }
  \end{equation}
  one has $\tau \mapsto \tau'$, because then chasing the commutative square of isomorphisms
  \begin{equation*}
    \xymatrix{ \Z \rho_{n-m-1} = H^{2n-2m-1, n-m}(W(n,m)) \ar[r] \ar[d] & \ar[d] H^{0,0}(Z_{n,m} \tau = \Z
      \tau  \\ 
      \Z \rho_{n-m-1} = H^{2n-2m-1, n-m}(W(n-1,m-1)) \ar[r] & H^{0,0}(Z_{n-1,m-1} \tau = \Z
      \tau'  }
  \end{equation*}
  we have $\rho_{n-m-1} \mapsto \rho_{n-m-1}$ as required.

  The difficulty is that the map $g:E_{n-1, m-1} \to E_{n,m}$ is a closed immersion, rather than a
  flat or split map, for which we have deduced this sort of naturality result in lemma \ref{l:techbff}.  We can however
  factor $g$ into such maps, which we denote only on the level of points, the obvious
  scheme-theoretic definitions are suppressed. Let $U_{n,m}$ denote the variety of $n,m$-matrices which (on the level
  of $\overline{k}$-points) have a decomposition as
  \begin{equation*}
    \begin{pmatrix} u & * & * \\ * & A & * \end{pmatrix}
  \end{equation*}
  where $u \in \overline{k}^*$, and $A \in W(n-1,m-2)(\overline{k})$. It goes without saying that
  this is a variety, since the conditions amount to the nonvanishing of certain minors. It is also
  easily seen that $U_{n,m}$ is an open dense subset of $E_{n,m}$. We have a map $E_{n-1,m-1} \to
  E_{n,m}$, given by
  \begin{equation*}
    B \mapsto \begin{pmatrix} 1  & *\\ * & B \end{pmatrix} 
  \end{equation*}
  and this map is obviously split by the projection onto the bottom-right $n-1 \times
  m-1$-submatrix. The composition $E_{n-1, m-1} \to U_{n,m} \to E_{n,m}$ is a factorization of the
  map $E_{n-1, m-1} \to E_{n,m}$ into a split map followed by an open
  immersion. The splitting of $E_{n-1, m-1}$ is a projection onto a factor, and since all schemes
  are flat over $\pt$, the splitting is flat as well. We may now use lemma \ref{l:techbff} twice to conclude that in diagram
  \eqref{e:mcchris} we have $\tau \mapsto  \tau'$, so that $\rho_{n-m+1} \mapsto \rho_{n-m+1}$, as
  asserted.
\end{proof}



\section{Higher Intersection Theory} \alabel{s:HIT}

We shall, as is standard, denote the algebraic $d$-simplex $\spec k[x_0 , \dots, x_d] / (x_0 + \dots
+ x_d -1) \isom \A^d$ by $\Delta^d$. The object $\Delta^\bullet$ is cosimplicial in an obvious
way. It shall be convenient later to identify $\Delta^1$ in particular with $\A^1 = \Spec k[t]$.

Let $X$ be a scheme of finite type over a field. The higher Chow groups of $X$, denoted $CH^i(X,d)$
are defined in \cite{BLOCH86} as the homology of a certain complex:
\begin{equation*}
  CH^i(X,d) = H_d(z^i(X,*))
\end{equation*}
where $z^i(X,d)$ denotes the free abelian group generated by cycles in $X \times \Delta^d$ meeting
all faces of $X \times \Delta^d$ properly. We denote the differential in this complex by $\delta$.

There is a comparison theorem, see \cite[lecture 19]{MVW}, \cite{ALLAGREE}, which states
that for any smooth scheme $X$ over any field $k$, there is an isomorphism between the motivic
cohomology groups and the higher Chow groups
\begin{equation*}
  CH^i(X,d) = H^{2i-d,i}(X,\Z)
\end{equation*}
or the equivalent with $\Z$ replaced by a general coefficient ring $R$. The
products on motivic cohomology and on higher Chow groups are known to coincide,
see\cite{WPRODAGREE}. In the difficult paper \cite{BLOCHMOVE}, the following result is proven in an
equivalent form (the strong moving lemma)
\begin{theorem}[Bloch] \alabel{t:BlochHardMove}
  Let $X$ be an equidimensional scheme, $Y$ a closed equidimensional subscheme of codimension $c$ in $X$, $U \isom X \sm Y$. Then for all $i$, there is an exact sequence of complexes
  \begin{equation*}
    \xymatrix{ 0 \ar[r] & z^{i-c}(Y, *) \ar[r] & z^i(X,*) \ar[r] & z^i_a(U,*) \ar[r] & 0}
  \end{equation*}
  where $z^i_a(U,*)$ is the subcomplex of $z^i(U,*)$ generated by subvarieties $\gamma$ whose
  closure $\overline{\gamma} \clsub X \times \Delta^*$ meet all faces properly. The inclusion of
  complexes $z_a^i(U,*) \subset z^i(U,*)$ induces an isomorphism on homology groups.
\end{theorem}

For a cycle $\alpha \in z^i(U,d)$, we can write $\alpha = \sum_{i=1}^N n_i A_i$ for some
subvarieties $A_i$ of $U \times \Delta^d$, and $n_i \in \Z \sm \{0\}$. We can form the
scheme-theoretic closure of $A_i$ in $X \times \Delta^d$, denoted $\overline{A_i}$. We remark that
$\overline{A_i} \times_{X\times\Delta^d} (U \times \Delta^d) = A_i$ \cite[II.3]{HART}. We define
\begin{equation*}
  \overline{\alpha} \sum_{i=1}^N n_i \overline{A_i}
\end{equation*}
We say that $\overline{\alpha}$ meets a subvariety $K \clsub X$ properly if every $\overline{A_i}$
meets $K$ properly. Suppose $\alpha$ is such that $\overline{\alpha}$ meets the faces of $X\times
\Delta^d$ properly, then $\alpha = (U\to X)^*(\overline{\alpha})$, so $\alpha \in z^i_a(U, d)$.

\begin{prop} \alabel{p:snakelemma}
  As before, let $X$ be a quasiprojective variety, let $Y$ be a closed subvariety of pure
  codimension $c$, let $U = X-Y$ and let $\iota: U \to X$ denote the open embedding. Suppose $\alpha
  \in z^i(U,d)$ is such that $\overline{\alpha}$ meets the faces of $X \times \Delta^d$ properly,
  then the connecting homomorphism $\bd: CH^i(U,d) \to CH^i(Y,d-1)$ takes the class of $\alpha$ to the class
  of $\delta(\overline{\alpha})$ which happens to lie in the subgroup $z^{i-c}(Y,d-1)$ of $z^{i-c}(X,d-1)$.
\end{prop}
\begin{proof}
  First, since $\alpha$ is such that $\overline{\alpha}$ meets the faces of $X\times \Delta^d$
  properly, it follows that $\alpha = \iota^*(\overline{\alpha})$, so $\alpha \in z^i_a(U , d)$.

  The localization sequence arises from the short exact sequence of complexes
  \begin{equation*}
    \xymatrix{ 0 \ar[r] & z^{i-c}(Y, *) \ar[r] & z^i(X,*) \ar[r] & z^i_a(U,*) \ar[r] & 0}
  \end{equation*}
  via the snake lemma.
A diagram chase now completes the argument.
\end{proof}

\begin{prop}\alabel{p:gentoothpaste}
  Consider $\A^n\sm\{0\}$ as an open subscheme of $\A^n$ in the obvious way, so there is a localization
  sequence in higher Chow groups for $\pt, \A^n$ and $\A^n \sm \{0\}$. The higher Chow groups
  $CH(\pt) = \M$ are given an explicit generator, $\nu$. Write $H^{2n-1,n}(\A^n\sm \{0\},\Z)=
  CH^n(\A^n\sm \{0\},1) = \Z\gamma \oplus Q$, where $Q=0$ for $n\geq 2$ and $Q= k^*$ for $n=1$, and
  where $\gamma$ is such that the boundary map
  \begin{equation*}
   \partial:  CH^n(\A^n \sm \{0\}, 1)  \to CH^0(\pt, 0)
  \end{equation*}
  maps $\gamma$ to $\nu$. The element $\gamma$ may be represented by any curve in
  \begin{equation*}
    \A^n\times \Delta^1 = \spec k[x_1,\dots, x_n,t]
  \end{equation*}
  which fails to meet the hyperplane $t=0$ and meets $t=1$
  with multiplicity one at $x_1=x_2 = \dots = x_n =0$ only.
\end{prop}
\begin{proof}
  The low-degree part of the localization sequence is
  \begin{equation*}
    \begin{split}
      \xymatrix@C=18pt{CH^0(\pt,1) = 0 \ar[r] & CH^n(\A^n,1) = Q \ar[r] & CH^n(\A^n \sm  0, 1) \ar[r]^{\bd} &
        CH^0(\pt,0)}\\
      \xymatrix{ \ar[r] & CH^0(\A^n,0) = \Z \ar[r] & CH^0(\A^n \sm 0,0) = \Z \ar[r] & 0}
    \end{split}    
  \end{equation*}
  Suppose $C$ is a curve which does not meet $t=0$, and which meets $t=1$ with multiplicity one at
  $x_1=\dots=x_n=0$ only, then by proposition \ref{p:snakelemma} the cycle $[C] \in CH^{n}(A^n-0,1)$
  maps to the class of a point in $CH^0(\pt, 0) = \Z$, which is a generator, \cite{FULTON}. The
  assertion now follows from straightforward homological algebra.
\end{proof}

\begin{cor}\alabel{p:curvetoothpaste}
  Suppose $p \in \A^n\sm \{0\}$ is a $k$-valued point. Write $p=(p_1,\dots,p_n)$. The curve given by the
  equation
  \begin{equation*}
    \gamma_p: \quad (x_1-p_1)t + p_1 = (x_2-p_2)t+p_2 = \dots = (x_n-p_n)t+p_n = 0
  \end{equation*}
  represents a canonical generator of $CH^n(\A^n\sm\{0\},1)$. 
\end{cor}
\begin{proof}
  One verifies easily that the proposition applies.
\end{proof}

\begin{cor} \alabel{p:multminusone}
  Consider the map $\A^n \sm \{0\} \to \A^n \sm \{0\}$ given by multiplication by $-1$. This map
  induces the identity on cohomology.
\end{cor}
\begin{proof}
  The preimage of the curve $\gamma_{p}$ is the curve $\gamma_{-p}$, but both represent the same
  generator of $CH^n(\A^n \sm \{0\}, 1)$, so the result follows.
\end{proof}

We can now prove two facts about the cohomology of $W(n,m)$ that should come as no surprise, the
case of complex Stiefel manifolds being our guide.

\begin{prop} \alabel{p:multminusoneW}
  Let $\gamma: W(n,m) \to W(n,m)$ denote multiplication of the first column by $-1$. Then $\gamma^*$ is
 identity on cohomology
\end{prop}
\begin{proof}
  By use of the comparison maps $\Gl(n) \to W(n,m)$, see proposition \ref{p:Wincl}, we see that it
  suffices to prove this for $\Gl(n)$. By means of the standard inclusion $\Gl(n-1) \to \Gl(n)$ and
  induction we see that it suffices to prove $\gamma^*(\rho_n) = \rho_n$, where $\rho_n$ is the
  highest-degree generator of $H^{*,*}(\Gl(n);R)$. By the comparison map again, we see that it
  suffices to prove that \[\tau: \A^n \sm \{0\} \to \A^n \sm \{0\}\] has the required property, but
  this is corollary \ref{p:multminusone}
\end{proof}

\begin{prop} \alabel{p:permutInvGln}
  Let $\sigma \in \Sigma_m$, the symmetric group on $m$ letters. Let $f_\sigma: W(n,m) \to W(n,m)$
  be the map that permutes the columns of $W(n,m)$ according to $\sigma$. Then $f_\sigma^*$ is the
  identity on cohomology.
\end{prop}
\begin{proof}
  We can reduce immediately to the case where $\sigma$ is a transposition, and from there we can
  assume without loss of generality that $\sigma$ interchanges the first two columns. Let $R$ be a
  finite-type $k$-algebra. We view $W(n,m)$ as the space whose $R$-valued points are $m$-tuples of
  elements in $R^n$ satisfying certain conditions which we do not particularly need to know. In the
  case $R = k$, the condition is that the matrix is of full-rank in the usual way.

  We can act on $W(n,m)$ by the elementary matrix $e_{ij}(\lambda)$
  \begin{equation*}
    e_{ij}(\lambda): (v_1, \dots, v_m, v_{m+1}) \mapsto (v_1, \dots, v_i + \lambda v_j, \dots, v_{m+1})
  \end{equation*}
  The two maps $e_{i,j}(\lambda)$ and $e_{i,j}(0) = \id$ are homotopic, so $e_{i,j}(\lambda)$
  induces the identity on cohomology. There is now a standard method to interchange two columns and
  change the sign of one by means of elementary operation $e_{ij}(\lambda)$, to wit
  $e_{12}(1)e_{21}(-1)e_{12}(1)$. We therefore know that the map
  \begin{equation*}
    (v_1, v_2, v_3, \dots, v_m) \mapsto (-v_2, v_1, v_3, \dots , v_m)
  \end{equation*}
  induces the identity on cohomology, but now proposition \ref{p:multminusoneW} allows us even to
  undo the multiplication by $-1$.
\end{proof}

\section{The Comparison Map: $\G_m \wedge \P^{n-1}_+ \to \Gl(n)$}

It will be necessary in this section to pay attention to basepoints. The group schemes $\G_m$ and
$\Gl_n$ will be pointed by their identity elements. When we deal with pointed spaces, we compute
reduced motivic cohomology for preference.

We establish a map in homotopy $\G_m \times \P^{n-1} \to \Gl(n)$, in fact we have a map from the
half-smash product
\begin{equation}
  \G_m\wedge\P^{n-1}_+ \to \Gl(n)
\end{equation}
 and we show this latter map induces isomorphism on a range of
cohomology groups.

We view $\P^{n-1}$ as being the space of lines in $\A^n$, and $\dual{\P}^{n-1}$ the space of
hyperplanes in $\A^n$. Define a space $\Pc^{n-1}$ as being the subbundle of $\P^{n-1} \times
\dual{\P}^{n-1}$ consisting of pairs $(L,U)$ where $L \cap U = 0$, or equivalently, where $L + U =
\A^n$. 

More precisely, we take $\P^{n-1}$ and construct 
\begin{equation*}
  \P^{n-1} \times \dual{\P}^{n-1} =\sProj_S(\mathcal{O}_S[y_0,\dots, y_{n-1}])
\end{equation*}
where $\sProj$ denotes the global projective-spectrum functor.  Let $Z$ denote the closed subscheme of $\P^{n-1} \times
\dual{\P}^{n-1}$ determined by the bihomogeneous equation $x_0y_0 + x_1y_1 + \dots + x_{n-1}y_{n-1}
= 0$. Then we define $\Pc^{n-1}$ as the complement $(\P^{n-1} \times \dual{\P}^{n-1}) \sm Z$.

\begin{prop}
  The composite $\Pc^{n-1} \to \P^{n-1} \times \dual \P^{n-1} \to \P^{n-1}$, the second map being
  projection, is a Zariski-trivializable bundle with fiber $\A^{n-1}$. In particular, $\Pc^{n-1}
  \weqto \P^{n-1}$
\end{prop}
\begin{proof}
 Taking $\A^{n-1}$ to be a coordinate open subscheme of $\P^{n-1}$ determined by e.g.~$x_0 \neq 0$, we obtain the
following pull-back diagram
\begin{equation*}
  \xymatrix{ U = \sProj_{\A^{n-1}}(\mathcal{O}_{\A^n}[y_0, \dots, y_{n-1}])\sm Z|_{\A^{n-1}} \ar[r] \ar[d] & \Pc^{n-1}
    \ar[d] \\ \A^{n-1} \ar[r] & \P^{n-1} }
\end{equation*}
The scheme $U$ is the complement of a hyperplane in $\P^{n-1} \times \dual \P^{n-1}$, and so takes the form 
\begin{equation*}
  U \isom \sSpec_{\A^{n-1}}(\mathcal{O}_{\A^n}[t_1, \dots, t_{n-1}]) \isom \A^{n-1} \times \A^{n-1}
\end{equation*}
The projection $U \to \A^{n-1}$ is a projection onto a factor. Since the coordinate open
subschemes $\A^{n-1}$ form an open cover of $\P^{n-1}$, it follows that $\Pc^{n-1} \weq \P^{n-1}$.
\end{proof}

In order to prove results concerning $\Pc^{n-1}$, it shall be useful to have the following definition
to hand. 
\begin{dfn}
  Let $R$ be a commutative $k$-algebra. By an \dfnnd{$n$-generated split line-bundle} we mean the
  following data. First, an isomorphism class of projective $R$-modules of rank $1$,  denoted $L$ by
  abuse of notation; second, a class of surjections $[f]:R^n \to L$, where two surjections are
  equivalent if they differ by a multiple of $R^\times$; third, a class of splitting maps $[g]: L \to
  R^n$, again the maps are considered up to action of $R^\times$, and where any $f' \in [f]$ is
  split by some $g' \in [g]$.
\end{dfn}

If $R \to S$ is a map of $k$-algebras, and if $(L, f,g)$ is an $n$-generated split line-bundle
over $R$, then application of $S \tensor_R \cdot$ yields the same over $S$. In this way, the
assignment to $R$ of the set of $n$-generated split line-bundles is a functor from the category
of $k$-algebras to the category of sets.

\begin{prop}
  If $R$ is a finite-type $k$-algebra, then the set of $R$-points $\Pc^{n-1}(R)$ is exactly the set
  of $n$-generated split line-bundles.
\end{prop}

\begin{proof}
  It is generally known that $\Spec R \to \P^{n-1}$ classifies isomorphism classes of rank-$1$
  vector bundles, $L$, over $R$, equipped with an equivalence class of surjections $R^n \to L$. This
  is a modification of a theorem of \cite[Chapter 4]{EGAII}. It follows that $\Spec R \to \P^{n-1}
  \times \dual \P^{n-1}$ classifies pairs of equivalence classes of rank-$1$ projective modules,
  equipped with surjective maps $(f:R^n \to L_1, g:\dual R^n \to L_2)$ considered up to scalar
  multiplication by $R^\times \times R^\times$.

  For convenience, let $\{e_1, \dots, e_n\}$ be a basis of $R^n$ and let $\{ \dual e_1, \dots, \dual
  e_n\}$ be the dual basis. Let $h: R \to R^n \tensor_R \dual R^n$ be the map given by the element
  $\sum_{i=1}^n e_i \tensor \dual e_i$ of the latter module. Given $f,g$, one obtains a composite map
  \begin{equation}
    \label{eq:778}
   \phi = (f \tensor g) \circ h :R \to L_1 \tensor_R L_2 \end{equation}
  Let $Z$ be the closed subscheme of $\P^{n-1} \times \dual{\P^{n-1}}$ determined by the equation
  $x_1y_1 + \dots + x_ny_n = 0$. Let $\m$ be a maximal ideal of $R$. Suppose that there is a dashed
  arrow making the following diagram commute:
  \begin{equation*}
    \xymatrix{ \Spec R/\m \ar[d] \ar@{-->}[r] & Z \ar[d] \\ \Spec R \ar[r] & \P^{n-1} \times
      \dual\P^{n-1} }
  \end{equation*}
  If we take the composite, $\Spec(R/\m) \to \Spec R \to \P^{n-1} \times \dual \P^{n-1}$, then this
  has the effect of reducing our represented maps modulo $\m$, and the result is two surjective maps
  over a field $\overline f: (R/\m)^n \to R/\m$ and $\overline g: (\dual R/ \m)^n \to \dual
  R/\m$. These can be identified with two $n$-tuples $[a_1; \dots ; a_{n-1}]$ and $[b_1; \dots ;
  b_{n-1}]$ one in $(R/\m)^n$, the other in its dual, taken up to multiplication by
  $(R/\m)^\times$. The map $\Spec R/\m \to \P^{n-1} \times \dual \P^{n-1}$ factors
  through $Z$ if and only if $a_1b_1 + \dots + a_n b_n = 0$, but this latter equation is precisely
  the statement that the reduction of the
  map $\phi$ of equation \eqref{eq:778} above, $(\overline f \tensor \overline g)  \cdot \overline h$, is nonzero.

  Since $R$ is a finite-type $k$-algebra, one has a factorization $\Spec R \to \Pc^{n-1} = (\P^{n-1}
  \times \dual \P^{n-1} )\sm Z \to \P^{n-1} \times \dual \P^{n-1}$ if and only if no closed point of
  $\Spec R$ lies in the closed subset $Z$ of $\P^{n-1} \times \dual \P^{n-1}$, but this is
  equivalent to the statement that no matter which maximal ideal $\m \subset R$ is chosen, $\Spec
  R/\m \to \P^{n-1} \times \dual \P^{n-1}$  does not factor through $Z$, and therefore to the statement that the reduction of $\phi$ at any maximal ideal of
  $R$ is never $0$. It follows that a map $\Spec R \to \P^{n-1} \times \dual \P^{n-1}$ is the data
  of two equivalence classes of line bundles, $L_1, L_2$, each equipped with surjective maps $f: R^n
  \to L_1$, $g: R^n \to L_2$, considered only up to scalar multiple, and where there is a
  nowhere-vanishing map of modules $\phi: R^1 \to L_1 \tensor_R L_2$. 

  Such a nowhere-vanishing map must be an isomorphism, $R^1 \isom L_1 \tensor_R L_2$, and we
  therefore have an identifying isomorphism $L_2 = \dual L_1$. Write $(a_1, \dots, a_n)$ for the
  image of the map $f: R^n \to L_1$ and $(b_1, \dots, b_n)$ for that of $g: \dual R^n \to \dual
  L_1$, then the identifying isomorphism has been constructed specifically so that the element
  $a_1b_1 + \dots + a_nb_n \in L_1 \tensor \dual L_1$ corresponds exactly to a generator of
  $R^1$, to wit.~a unit $u \in R$. Since we are working only with equivalence-classes of presentations of $L_1$, $L_2$, we may
  if need be replace $L_1$ by $u^{-1}L_1$, and so we find that the composite
  \begin{equation*}
    \xymatrix{ L_1 \ar^{\dual g}[r] & R^n \ar^f[r] & L_1 }
  \end{equation*}
  is the identity, as required.
\end{proof}

\begin{prop}
For all $n \in \N$, let $X_n$ denote the motivic space
\begin{equation} \alabel{e:decompOfXn}
  X_n = \G_m \wedge (\Pc^{n-1}_+) 
\end{equation}
Then there are maps $f_n:X_n \to \Gl(n)$, and maps $\overline{h} : X_n \to X_{n+1}$, which make the following diagram commute
\begin{equation} \alabel{e:CommutingForXn}
  \xymatrix{  \G_m \wedge \P^{n-1}_+ \ar^{\id\wedge h_+}[d] & \ar[l] X_n \ar^{\overline{h}}@{-->}[d]\ar[r] & \Gl(n) \ar^\phi[d] \\ \G_m
    \wedge \P^n_+ & \ar[l] X_{n+1} \ar[r] & \Gl(n+1) }
\end{equation}
Where the maps $h: \P^{n-1} \to \P^n$ are the standard inclusion of $\P^{n-1} \to \P^n$ as the first
$n-1$ coordinates.
\end{prop}
\begin{proof}
  First we construct a map $\overline{f}_n : \G_m \times X_n \to \Gl(n)$. The first scheme represents the functor taking
  a finite-type $k$-algebra, $R$, to the set of elements of the form $(\lambda, (L, \phi, \psi))$,
  that is to say, consisting of elements $\lambda \in
  R^\times$ and a surjection onto a rank-1 projective bundle $\phi:R^n
  \to L$ along with a splitting $\psi: L \to R^n$. The second scheme, $\Gl(n)$, represents the functor taking $R$ to
  $\Gl_n(R)$. The map $f_n$ can be constructed therefore as a natural transformation between functors. We set up such a
  transformation as follows: the data $\phi$, $\psi$ amount to an isomorphism 
  \begin{equation*}
    \Phi: \ker \phi \oplus L \isomto R^n
  \end{equation*}
  We can define a map 
  \begin{equation*}
    \xymatrix{ \Phi_\lambda:  R^n \isom \ker \phi \oplus L \ar^{(\id, \lambda)}[rr] & & \ker \phi
      \oplus L \isom R^n }
  \end{equation*}
 it has inverse
  $(\id, \lambda^{-1})$, so it is an automorphism, and consequently an element of $\Gl_n(R)$. The transformation taking $(\lambda, (L, \phi, \psi))$ to $\Phi_\lambda$ is natural, and so, by Yoneda's lemma, is a
  map of schemes.

  The motivic space $\G_m \wedge \Pc^{n-1}_+$ is the sheaf-theoretic quotient of the inclusion of sheaves of simplicial
  sets $1 \times \Pc^{n-1} \to \G_m \times \Pc^{n-1}$, in particular, $\overline{f}_n$ will descend to a map $f_n : \G_m
  \wedge \Pc^{n-1}_+$ if and only if the composite $1 \times \Pc^{n-1} \to \G_m \times \Pc^{n-1} \to \Gl(n)$ is
  contraction to a point. On the level of functors, however, the first scheme represents the functor taking $R$ to pairs
  $(1, (L, \phi, \psi))$, and it is immediate that $\Phi_1 = \id$, so that the composite is indeed the constant map at
  the identity of $\Gl(n)$. We have furnished therefore the requisite map $\G_m \wedge \Pc^{n-1}_+$. 

  As for the commutativity of the diagram, we recall that the standard inclusion $\P^{n-1} \to \P^n$ represents the
  natural transformation taking a surjection such as $R^N \to L$ to the trivial extension $R^n \oplus R \to R^n \to
  L$. We lift this idea to $\Pc^{n-1}$, given a triple $(L, \phi, \psi) \in \Pc^{n-1}(R)$, one can extend the split maps
  $\phi, \psi$ to maps $\overline \phi$ and $\overline \psi$, where $\overline \phi : R^{n+1} \to L$ and
  $\overline{\phi}$ splits $\overline{\phi}$, simply by adding a trivial summand to $R^n$. This furnishes a natural
  transformation of functors, or a map of schemes, $\Pc^{n-1} \to \Pc^n$, that makes the diagram
  \begin{equation*}
    \xymatrix{ \P^{n-1} \ar^h[d] \ar[r] & \Pc^{n-1} \ar[d] \\ \P^n \ar[r] & \Pc^n }
  \end{equation*}
  commute. One may define $\overline{h} : X_n \to X_{n+1}$ in diagram \eqref{e:CommutingForXn} as the map given by
  application of $\G_m \wedge ( \cdot )_+$ to the map $\Pc^{n-1} \to \Pc^n$ immediately constructed.

  The map $\phi: \Gl(n) \to \Gl(n+1)$ is obtained as a natural transformation by taking $A \in \Gl_n(R)$ and
  constructing $A \oplus \id: R^n \oplus R \to R^n \oplus R$. It is now routine to verify that diagram
  \eqref{e:CommutingForXn} commutes.
  
\end{proof}

The construction $\G_m \wedge X$ is denoted by $\Sigma_t^1 X$ and is called the Tate suspension,
\cite{VREDPOWER}. We have
\begin{equation*}
  H^{*,*}(\G_m;R) \isom \frac{\M[\sigma]}{\sigma^2 - \{-1\}\sigma}
\end{equation*}
the relation being derived in loc.~cit. It is easily seen that as rings, we have
\begin{equation*} 
H^{*,*}(\G_m \times X; R) \isom H^{*,*}(X;R) \tensor_\M \frac{\M[\sigma]}{(\sigma^2 - \{-1\}\sigma)}
\end{equation*}
and that $\tilde{H}(\Sigma_t^1 X;R)$ is the split submodule (ideal) generated by $\sigma$, leading to
a peculiar feature of the Tate suspension
\begin{prop}
  Suppose $x, y \in \tilde H^{*,*}(X;R)$, and that $\sigma x, \sigma y$ are their isomorphic images
  in $\tilde H^{*,*}(\Sigma_t^1X;R)$. Then $(\sigma x)(\sigma y) = \{-1\}\sigma (xy)$.
\end{prop}

We now come to the main theorem of this paper

\begin{theorem} \alabel{t:MainComp}
  The map $f_n$ induces an isomorphism on cohomology 
\begin{equation*}
H^{2j-1,j}(\Gl(n);R) \to H^{2j-1,j}(\Sigma_t^1 \Pc^{n-1}_+) = H^{2j-1,j}(\Sigma_t^1\P^{n-1}_+)
\end{equation*}
 in dimensions $(2j-1, j)$ where $j\ge 1$.
\end{theorem}
\begin{proof}
  We first remark that when $j$ is large, $j>n$, 
\begin{equation*}
  H^{2j-1,j}(\Gl(n);R) = H^{2j-1,j}(\Sigma_t^1  \P^{n-1};R) = 0
\end{equation*}
 so the result holds trivially in this range. We restrict to the case $1 \le j
  \le n$.

  The following are known:
  \begin{align*} H^{*,*}(\Gl_n) & = \M[\rho_n, \dots, \rho_1]/I \quad |\rho_i|=(2i-1,i)\\
    \tilde H^{*,*}(\Sigma_t^1 \Pc^{n-1}_+) = \tilde H^{*,*}(\Sigma_t^1\P^{n-1}_+)  &=
  \bigoplus_{i=0}^{n-1} \sigma\eta^i \M \quad |\sigma| = (1,1), \quad |\eta| = (2,1)\end{align*}

It suffices to show that $f_n^*(\rho_i) = \sigma\eta^{i-1}$. Since $f_n^*(\rho_i) = \tilde
f_n^*(\rho_i)$, we may prove this for the map $\tilde f_n : \G_m \times \Pc^{n-1}$, which has the benefit of being more explicitly
geometric.
  
We prove this by induction on $n$. In the case $n=1$, the space $\Pc^{n-1}$ is trivial, and $X_n = \G_m = \Gl(n)$. The
map $f_1$ is the identity map, so the result holds in this case.

  There is a diagram of varieties
\begin{equation}
  \xymatrix{  \G_m \wedge \P^{n-1}_+ \ar^{\id\wedge h_+}[d] & \ar[l] X_n \ar^{\overline{h}}[d]\ar[r] & \Gl(n) \ar^\phi[d] \\ \G_m
    \wedge \P^n_+ & \ar[l] X_{n+1} \ar[r] & \Gl(n+1) }
\end{equation}
  as previously constructed. We understand the vertical map on the left since we can rely on the theory of
  ordinary Chow groups, \cite[chapter 1]{FULTON}, we know that the induced map $H^{2j,j}(\P^{n-1})
  \to H^{2j,j}(\P^{n-2})$ is an isomorphism for $j<n-1$, and so $i^*$ is an isomorphism
  $H^{2j-1,j}(\G_m \times \Pc^{n-1}) \isom H^{2j-1,j}(\G_m\times \Pc^{n-2})$ for $j<n$.

  The maps $\phi_i^*$ are also isomorphisms in this range, by proposition \ref{p:GLNsurj} and its
  corollary, so the diagram implies that the result holds except possibly for $f_n^*(\rho_n)$.

  \medskip

  The argument we use to prove $f_n^*(\rho_n) = \sigma \eta^{n-1}$ is based on the composition
  \begin{equation*}
    \xymatrix{ \G_m \times \Pc^{n-1} \ar[r] \ar^g@/^2ex/[rr] & \Gl(n) \ar_{\pi}[r] & \A^n\sm\{0\}}
  \end{equation*}
  where the map $\pi$ is projection on the first column. We write $g$ for the composition of the two
  maps. Since $H^{*,*}(\A^n\sm\{0\}) \isom \M [\iota]/(\iota^2)$, with $|\iota| = (2n-1,n)$, and
  $\pi^*(\iota) = \rho_n$, it suffices to prove that $g^*(\iota) = \sigma \eta^{n-1}$.

  For the sake of carrying out computations, it is helpful to identify motivic cohomology and higher
  Chow groups, e.g.~identify $H^{2n-1,n}(\A^n\sm\{0\})$ and $CH^n(\A^n\sm\{0\},1)$. We can write down an
  explicit generator for $CH^n(\A^n\sm\{0\},1)$, see corollary \ref{p:curvetoothpaste}, for instance the
  curve $\gamma$ in $\A^n\sm\{0\} \times \Delta^1$ given by $t(x_1-1)=-1$, $x_2=x_3=\dots=x_n= 0$. We can
  also write down an explicit generator, $\mu$, for a class $\sigma \eta_0^{n-1} \in \G_m \times \P^{n-1}$,
  writing $x$ and $a_0, \dots, a_{n-1}$ for the coordinates on each, and $t$ for the coordinate
  function on $\Delta^1$, $\sigma \eta^{n-1}$ is explicitly represented by the product of the
  subvariety of $\P^{n-1}$ given by $a_0=1,a_2 = 0 \dots, a_{n-1}=0$, which represents $\eta^{n-1}$,
  with the cycle given by $t(x-1) = -1$ on $\G_m \times \Delta^1$.

  We have therefore a closed subvariety, $\gamma$, in $\A^n \sm \{0\} \times \Delta^1$, and another closed subvariety
  $\mu \in \G_m \times \P^{n-1}$, representing the cohomology classes we wish to relate to one
  another. We shall show that the pull-backs of each to $\G_m \times \Pc^{n-1} \times \Delta^1$ coincide. For this, it shall again be
  advantageous to take the functorial point of view.

  The variety $\A^n \sm \{0\} \times \Delta^1$ represents, when applied to a $k$-algebra $R$, unimodular columns $(x_1,
  \dots, x_n)^t \in R^n$ of height $n$, along with a parameter $t \in R$. The subvariety $\gamma$ represents the
  unimodular columns and parameters for which $t(x_1-1) = -1$ and $x_i = 0$ for $i > 1$. Note that for such pairs, we
  have $t \in R^\times$ and $t \neq 1$ (since otherwise $x_1 =0$).

  The pull-back of $\gamma$ to $\Gl(n) \times \Delta^1$ represents pairs $(A, t)$, where $A \in \Gl_n(R)$ and $t \in R$,
  satisfying $(A(1, 0, \dots, 0)^t, t) \in \gamma(R)$. We observe that, writing $e_1 = (1, 0, \dots, 0)^t$, this implies
  $Ae_1 = (1 - t^{-1})e_1$. Note further that $1- t^{-1} \neq 1$.

  Pulling $\gamma$ back a second time, to $\G_m \times \Pc^{n-1} \times \Delta^1$, we obtain the set of triples
  $(\lambda,  (L, \phi, \psi), t)$, where $\lambda \in R^\times$, $(L, \phi, \psi)$ is a split $n$-generated line
  bundle, and $t$ is a parameter, and where the invertible linear transformation $\Phi_\lambda : \ker \phi \oplus L \to
  \ker \phi \oplus L$ along with the parameter $t$ lies in the pull-back of $\gamma(R)$ to $\Gl_n(R)$. Decomposing $e_1
  = v+ w$, where $v \in \ker \phi$ and $w \in L$, we see that $\Phi_L(e_1) = v + \lambda w = (1-t^{-1})e_1 =
  (1-t^{-1})(v +w)$, which by uniqueness of the decomposition, forces $\lambda = (1-t^{-1})$, and $v=0$. Consequently,
  $L$ is the rank-$1$ split subbundle of $R^n$ generated by $e_1$, and we have $t (\lambda - 1) =  - 1$.

  On the other hand, the variety $\mu \subset \G_m \times \P^{n-1} \times \Delta^1$ represents triples $(\lambda, L, t)
  \in R^\times \times \P^{n-1}(R) \times R$, where $L$ is exactly the rank-$1$ free module generated by $e_1$, and where
  $t(\lambda -1) = -1$. The pull-back of $\mu$ to $\G_m \times \Pc^{n-1} \times \Delta^1$ coincides
  with that of $\gamma$, as claimed.
 \end{proof}

We are now in a position to pay off at last the debt we owe regarding the product structure of $H^{*,*}(W(n,m);R)$.
\begin{thm} \alabel{th:mainch1}
   The cohomology of $W(n,m)$ has the following presentation as a graded-commutative $\M_R$-algebra:
  \begin{equation*}
    H^{*,*}(W(n,m);R) = \frac{\M_R[\rho_n,\dots, \rho_{n-m+1}]}{I} \qquad |\rho_i| = (2i-1,i)
  \end{equation*}
  The ideal $I$ is generated by relations $\rho_i^2-\{-1\}\rho_{2i-1}$, where $\{-1\} \in
  \M_R^{1,1}$ is the image of $-1 \in k^* = \M_\Z$ under the map $\M_\Z \to \M_R$.
\end{thm}
\begin{proof}
  It suffices to deal with the case $R= \Z$. It suffices also to consider only the case $m=n$,
  since we can use the inclusion $H^{*,*}(W(n,m)) \subset H^{*,*}(\Gl(n))$ to deduce it for all $n,m$.

  We have proved everything already in proposition \ref{p:CohW1}, except that in the relation
  $\rho_i^2 - a\rho_{2i-1}$, we were unable to show $a$ was nontrivial. We consider the map $\G_m
  \times \Pc^{n-1}_+ \to \Gl(n)$, which induces a map of rings on cohomology. In the induced map, we
  have $\rho_i \mapsto \sigma \eta^{i-1}$, and so $\rho_i^2 \mapsto -\sigma^2\eta^{2i-2} =
  \{-1\}\sigma \eta^{2i-2}$. Since this is nontrivial if $2i-2 \le n-1$, it follows that $\rho_i^2$
  is similarly nontrivial.
\end{proof}

In the case $n=m$ this result, although computed by a different method, appears in  \cite{PUSHIN}.

We can compute the action of the reduced power operations of \cite{VREDPOWER} on the cohomology
$H^{*,*}(W(n,m);\Z/p)$ by means of the comparison theorem.

\begin{theorem} \alabel{c:SqOnW} Suppose the ground-field $k$ has characteristic different from $2$.
  Represent $H^{*,*}(W(n,m); \Z/2)$ as $\M_2[\rho_n,\dots,\rho_{n-k+1}]/I$. The even motivic
  Steenrod squares act as
  \begin{equation*}
    \Sq^{2i}(\rho_j) =  \begin{cases}\binom{j-1}{i} \rho_{j+i}\quad \text{ if $i+j \leq
        n$} \\ 0 \quad \text{otherwise} \end{cases}
  \end{equation*}
  The odd squares vanish for dimensional reasons.
\end{theorem}

\begin{theorem} \alabel{c:POnW}
  Let $p$ be an odd prime and suppose the ground-field has characteristic different from
  $p$. Represent $H^{*,*}(W(n,m); \Z/p)$ as $\M_p[\rho_n, \dots, \rho_{n-k+1}]/I$. The reduced power
  operations act as
  \begin{equation*}
    P^i(\rho_j) = \begin{cases} \binom{j-1}{i} \rho_{ip+j-i} \quad \text{ if $ip+j-i \le n$} \\ 0
      \quad \text{otherwise} \end{cases}
  \end{equation*}
  The Bockstein vanishes on these classes for dimensional reasons.
\end{theorem}

Observe that in both cases, since the cohomology ring is multiplicatively generated by the $\rho_j$,
the given calculations suffice to deduce the reduced-power operations in full on the appropriate
cohomology ring.

\begin{proof}
  We prove only the case of $p=2$, the other cases being much the same.

  We observe that $\Sq^{2j}$ is honest squaring on
  $H^{2j,j}(\P^n;\Z/2)$, on the classes $\eta^i \in H^{2i,i}(\P^n;\Z/2)$ the Bockstein vanishes, and
  as a consequence the expected Cartan formula obtains for calculating $\eta^{i+i'}$, it is a simple
  matter of induction to show that $\Sq^{2i}(\theta^j) = \binom{j}{i} \theta^{j+i}$.

  There is an inclusion of $H^{*,*}(W(n,m);\Z/2) \subset H^{*,*}(\Gl(n);\Z/2)$ arising from the projection map,
  see proposition \ref{p:Wincl}. It suffices therefore to compute the action of the squares on
  $H^{*,*}(\Gl(n);\Z/2)$. Using the previous proposition and the decomposition in equation
  \eqref{e:decompOfXn}, we have isomorphisms 
  \begin{equation*}
    H^{2n-1,n}(\Gl(n);\Z/2) \isom H^{2n-1,n}(\Sigma_t^1\P^{n-1};\Z/2)
  \end{equation*}
  The reduced power operations are stable not only with respect to the simplicial suspension, but
  are also stable with respect to the Tate suspension. This allows a transfer of the calculation on
  $\P^n$ to the calculation on $\GL_n$ via the comparison of theorem \ref{t:MainComp}

  To be precise, we have
  \begin{equation*} \begin{split}
    f_n^*(\Sq^{2i}\rho_j) = \Sq^{2i}f_n^*(\rho_j) = \Sq^{2i}\sigma\eta^{j-1}  \\
    =\sigma\Sq^{2i}\eta^{j-1} = \begin{cases} \binom{j-1}{i}\sigma\eta^{j+i-1}=
      \binom{j-1}{i}f_n^*(\rho_{j+i}) \quad 
      \text{if $i+j \leq n$} \\ 0 \quad \text{otherwise} \end{cases} \end{split}
  \end{equation*}
  Since $f_n^*$ is an isomorphism on these groups, the result follows.  
\end{proof}

\bibliography{BThesis}
\bibliographystyle{alpha}

\end{document}

%% file: MC.tex
\input{algebraHeader}

%% file: algebraHeader.tex
\input{generalHeader}

\newcommand{\Gl}{\operatorname{GL}}

\newcommand{\colim}{\operatornamewithlimits{colim}}
\newcommand{\cat}[1]{\text{\bf #1}}

\newcommand{\Pre}{\cat{Pre}}
\newcommand{\tensor}{\otimes}

\newcommand{\spec}{\operatorname{Spec}}

\newcommand{\Spec}{\spec}

\newcommand{\Sm}{\textit{Sm}}

\newcommand{\pt}{\text{\rm pt}}

\newcommand{\clsub}[1][{}]{\overset{#1}{\rightarrowtail}}

\newcommand{\Nis}{\text{\it Nis}}

\newcommand{\A}{\mathbb{A}}
\renewcommand{\P}{\mathbb{P}}

\newcommand{\dual}[1]{\check{#1}}

\newcommand{\G}{\mathbb{G}}
\newcommand{\lspan}[1]{\left<#1\right>}

\newcommand{\M}{\mathbb{M}}

\newcommand{\weq}{\simeq}

\newcommand{\weqto}{\overset{\sim}{\to}}

\newcommand{\Sq}{\operatorname{Sq}}

\providecommand{\deg}{\operatorname{deg}}
\providecommand{\weight}{\operatorname{wt}}

\newcommand{\sProj}{\mathcal{P}\!\operatorname{roj}}
\newcommand{\sSpec}{\mathcal{S}\!\operatorname{pec}}

\newcommand{\m}{\mathfrak{m}}

\newcommand{\GL}{\Gl}

%% file: generalHeader.tex
\usepackage{amssymb, amsthm, amsmath}
\usepackage{pdfsync}

\newcommand{\id}{\mathrm{id}}
\newcommand{\isom}{\cong}

\newcommand{\mto}[1]{\stackrel{#1}{\longrightarrow}}
\newcommand{\isomto}{\mto{\isom}}

\newcommand{\sm}{\setminus}

\newcommand{\C}{\mathbb{C}}
\newcommand{\N}{\mathbb{N}}

\newcommand{\R}{\mathbb{R}}
\newcommand{\Z}{\mathbb{Z}}

\newcommand{\bd}{\partial}

\newtheorem{thm}{Theorem}
\newtheorem{lemma}[thm]{Lemma}
\newtheorem{theorem}[thm]{Theorem}
\newtheorem{prop}[thm]{Proposition}
\newtheorem{cor}{Corollary}[thm]

\newenvironment{dfn}{\null\par \textbf{Definition:}}{ \null \par}

\newcommand{\anote}[1]{} 
\newcommand{\alabel}[1]{\label{#1} } 
\newcommand{\dfnnd}[1]{\textit{#1}}

%% file: GLN.bbl
\begin{thebibliography}{MVW06}

\bibitem[Ada62]{ADAMSVBL}
J.~F. Adams.
\newblock Vector fields on spheres.
\newblock {\em Ann. of Math. (2)}, 75:603--632, 1962.

\bibitem[Blo86]{BLOCH86}
Spencer Bloch.
\newblock Algebraic cycles and higher {$K$}-theory.
\newblock {\em Advances in Mathematics}, 61:267--304, 1986.

\bibitem[Blo94]{BLOCHMOVE}
Spencer Bloch.
\newblock The moving lemma for higher {C}how groups.
\newblock {\em Journal of Algebraic Geometry}, 3:537--568, 1994.

\bibitem[Del]{DELIGNE}
Pierre Deligne.
\newblock Lectures on motivic cohomology 2000/2001.
\newblock www.math.uiuc.edu/K-theory/0527.

\bibitem[DHI04]{DHI}
Daniel Dugger, Sharon Hollander, and Daniel~C. Isaksen.
\newblock Hypercovers and simplicial presheaves.
\newblock {\em Math. Proc. Cambridge Philos. Soc.}, 136(1):9--51, 2004.

\bibitem[Ful84]{FULTON}
William Fulton.
\newblock {\em Intersection Theory}.
\newblock Springer Verlag, 1984.

\bibitem[Gro61]{EGAII}
A.~Grothendieck.
\newblock \'{E}l\'ements de g\'eom\'etrie alg\'ebrique. {II}. \'{E}tude globale
  \'el\'ementaire de quelques classes de morphismes.
\newblock {\em Inst. Hautes \'Etudes Sci. Publ. Math.}, (8):222, 1961.

\bibitem[Har77]{HART}
Robin Hartshorne.
\newblock {\em Algebraic Geometry}, volume~52 of {\em Gradute Texts in
  Mathematics}.
\newblock Springer Verlag, 1977.

\bibitem[Jam76]{JAMES}
I.~M. James.
\newblock {\em The Topology of Stiefel Manifolds}.
\newblock Cambridge University Press, Cambridge, England, 1976.

\bibitem[MV99]{MV}
Fabien Morel and Vladamir Voevodsky.
\newblock $\mathbb{A}^1$-homotopy theory of schemes.
\newblock {\em Inst. Hautes \'Etudes Sci. Publ. Math.}, (90):45--143 (2001),
  1999.

\bibitem[MVW06]{MVW}
Carlo Mazza, Vladimir Voevodsky, and Charles Weibel.
\newblock {\em Lectures on Motivic Cohomology}, volume~2 of {\em Clay
  Monographs in Math}.
\newblock AMS, 2006.

\bibitem[Pus04]{PUSHIN}
Oleg Pushin.
\newblock Higher {C}hern classes and {S}teenrod operations in motivic
  cohomology.
\newblock {\em $K$-Theory}, 31(4):307--321, 2004.

\bibitem[Ray68]{MRaynaud}
Mich{\`e}le Raynaud.
\newblock Modules projectifs universels.
\newblock {\em Invent. Math.}, 6:1--26, 1968.

\bibitem[Voe02]{ALLAGREE}
Vladimir Voevodsky.
\newblock Motivic cohomology are isomorphic to higher {C}how groups in any
  characteristic.
\newblock {\em International Mathematics Research Notices}, 7:351--355, 2002.

\bibitem[Voe03]{VREDPOWER}
Vladimir Voevodsky.
\newblock Reduced power operations in motivic cohomology.
\newblock {\em Publications Math\'ematiques de l'IH\'ES}, 98:1--57, 2003.

\bibitem[Wei99]{WPRODAGREE}
Charles~A. Weibel.
\newblock Products in higher {C}how groups and motivic cohomology.
\newblock In {\em Algebraic $K$-theory, Seattle WA 1997}, Proceedings of
  Symposia in Pure Mathematics, pages 305--315. American Mathematical Society,
  1999.

\end{thebibliography}
